\documentclass[11pt]{article}
\usepackage{geometry}

\usepackage[square,sort,comma,numbers]{natbib}
\usepackage{epsfig}
\usepackage{float}
\usepackage{eucal}
\usepackage{appendix}
\usepackage{amssymb}
\usepackage{amsfonts}
\usepackage{amsmath}
\usepackage{amstext}
\usepackage{amsthm}
\usepackage{floatflt}
\usepackage{nicefrac}
\usepackage{bbm}
\usepackage{dsfont}

\include{macros_nicole}

\newtheorem{prop}{Proposition}
\newtheorem{theo}{Theorem}
\newtheorem{assumption}{Assumption}

\newtheorem{lemma}{Lemma}

\usepackage{hyperref}

\title{Reducing training time by efficient localized kernel regression}
\author{Nicole M\"ucke\footnote{Institute for Stochastics and Applications, University of Stuttgart, {\it nicole.muecke@mathematik.uni-stuttgart.de} } }
\date{}

\begin{document}
\maketitle

\begin{abstract}
We study generalization properties of kernel regularized least squares regression based on a partitioning approach. We show that 
optimal rates of convergence are preserved  if the number of local sets grows sufficiently slowly with the sample size. 
Moreover, the partitioning approach can be efficiently combined with local Nystr\"om subsampling, improving computational cost twofold. 
\end{abstract}

\section{Introduction}

The use of reproducing kernel methods for non-parametric regression such as {\it Kernel Regularized Least Squares} (KRLS) or the Support Vector Machine 
has enjoyed a wide popularity and  their theoretical properties are well understood. 
These methods are attractive because they  attain asymptotically 
minimax optimal rates of convergence. But it is also well known that they scale poorly when massive datasets are involved. 
Large training sets give rise to large computational and storage costs. 
For example, computing a kernel ridge regression estimate needs inversion of a $n \times n$- matrix, with $n$ the sample size. This requires ${\cal O}(n^3)$ time and ${\cal O}(n^2)$ 
memory, which becomes prohibitive for large sample sizes. 
\\
\\
{\bf Large Scale Problems: Subsampling and Localization. } 
Because of the above mentioned shortcomings various methods have been developed for saving computation time and memory requirements, speeding up the 
usual approaches. 
During the last years, 
a huge amount of research effort was devoted to finding {\it low-rank approximations} of the kernel matrix.  
A popular instance is Nystr\"om sampling see e.g. \cite{WillSee00}, \cite{Bach2013}, \cite{RudCamRos15}\; where one aims at  
replacing the theoretically optimal approximation obtained by a spectral decomposition (which requires time at least ${\cal O}(n^2)$) 
by a less ambitious suitable low rank approximation  of the kernel matrix via column sampling, reducing
run time to ${\cal O}(np^2)$  where $p$ denotes the rank of the approximation. Clearly the rules of the game are to choose $p$ as small as possible while 
maintaining minimax optimality of convergence rates  
and to explicitly determine this $p$ as a function of the sample size $n$.   

Another line of research with computational benefits is devoted to so called {\it partition-based} or {\it localized} approaches, 
see \cite {SegBla10} for localized SVMs for binary classification, \cite{EbStein15} for localized SVMs using the Gaussian RBF kernel or \cite{Tandon16} 
for more general kernels in an KRLS framework. The main idea behind 
the partitioning approach is to  split the training data based on a disjoint partition of the input space into smaller subsamples and  to 
train only on smaller chunks. 
Prediction for a new input is then much faster since one only has to identify the 
local subset to which the new input belongs and to use the local estimator.  

Another benefit in using localized approaches lies in exploiting regions of high regularity. 
It is well known that rates of convergence highly depend on regularity: The smoother the objective function, the faster the rate of convergence. 
The usual global learning approach however doesn't "see" regions of higher regularity. Global rates of convergence are determined by the region of the input 
space where the target is least smooth. 

Our results show, when building an KRLS estimator based on accurate local ones trained on subregions of the training set, we better take into account the local regularity 
of the objective function, leading to more accurate local approximations. In particular, our approach does not suffer from local  
underfitting, even though the regularization parameter is chosen as in the global approach.

Further, we show that the partitioning approach for KRLS can be efficiently combined with Nystr\"om subsampling, 
substantially reducing training time and speeding up the more usual (localized) version of KRLS.  

Informally, we show if the number of subsets is {\it not too large} and if the number of subsampled datapoints is {\it large enough} 
we obtain  fast upper rates of convergence. 
An important aspect of our approach is the observation that under appropriate conditions on the probability of subsamples
- which come quite naturally in the partitioning approach - our rates of convergence are actually guided by local regions of high regularity, 
leading to improved finite sample bounds.   
 
In this paper, we shall focus only on KRLS, 
although our results could be extended to a much larger class of general spectral regularization methods, including e.g. Gradient Descent, 
similar to \cite{rosascotechrep}\;, \cite {DicFosHsu15}\;, \cite{BlaMuc16} or more recent \cite{Lin18}. 
For a more detailed discussion of our results and a comparison to related research we refer to Section \ref{sec:distributed}. 
\\
\\
The outline of our paper is as follows: 
Section \ref{sec:learning} is devoted to an introduction to the learning problem in an RKHS framework.   
In Section \ref{sec:Partitioning_Approach} we firstly introduce the partitioning approach and introduce all 
assumptions needed to establish our 
main Theorems. In Section \ref{sec:KRR_Nys} we briefly recall the Nyst\"om method and give an upper bound 
in expectation for the rate of convergence. Section \ref{sec:localnysation} is devoted to showing that the partitioning approach and 
subsampling can be effeciently combined. Finally, we compare our results with other approaches in Section \ref{sec:distributed} and  finish with a conclusion 
in Section \ref{sec:conclusion}\;. All our proofs a deferred to the Appendix. 
\\
\\
{\bf Notation: } For $n \in \mbn$, we denote by $[n]$ the set of integers $\{1,...,n\}$. For  two positive sequences $(a_n)_n$ and 
$(b_n)_n$, the expression $a_n \lesssim b_n$ means that $a_n \leq C b_n$, for some universal constant $C<\infty$.  
For $f$ in a Hilbert space $\cH$ we let $f\otimes f$ be the outer product 
acting as rank-one operator $(f \otimes f)h = \inner{h,f}_\cH f$.

\section{Learning with Kernels}
\label{sec:learning}

In this section we introduce the supervised learning problem and give an overview of regularized learning in an RKHS framework. 
\\
\\
{\bf Learning Setting.} 
We consider the well-established setting
of learning under random design where $\cX \times \mbr$ is a probability space with distribution $\rho$. We let $\nu$ be the marginal distribution 
on $\cX$ and $\rho(\cdot|x)$ denotes the conditional distribution on $\mbr$ given $x \in \cX$.  Our goal is minimizing the 
{\it expected risk}
\[ \cE(f) = \int_{\cX \times \mbr} (f(x)- y)^2\; d\rho(x,y) \;. \]
It is known that this quantity is minimized over $L^2(\nu)$ by the regression function 
\[ f_\rho(x) = \int_\cY y \; d\rho(y|x) \;.\] 
However, we exclusively focus our analysis to the special case where $f_\rho$ lies in a hypothesis space $\cH \subset L^2(\nu)$ of 
measurable functions from $\cX$ to $\mbr$.

We are particularly interested in the case where $\cH$ is a separable {\it reproducing kernel Hilbert space} (RKHS), possessing a 
bounded positive definite symmetric measurable kernel $K$ on $\cX$. 
Throughout the paper we assume that
\begin{assumption}
\[ \kappa^2 := \sup_{x,x'} K(x,x') < \infty \;. \]
\end{assumption}
An important feature is the {\it reproducing property:} For any $x \in \cX$ and any $f \in \cH$ one has 
\[  f(x)=\inner{f, K_x}_\cH \;,\]
where $K_x:=K(x,\cdot) \in \cH$, see e.g. \cite{aron50}.

Given a sample $\bz=\{z_j=(x_j, y_j)\}_{j=1}^n$ of size $n \in \mbn$, a classical approach for empirically solving the minimization problem described above is by 
{\it Kernel Regularized Least Squares} (KRLS), also known as {\it Tikhonov Regularization}. This approach is based on minimization  of the penalized empirical functional 
\begin{equation}
\label{eq:min}
  \min_{f \in \cH}\; \frac{1}{n}\sum_{j=1}^n (f(x_j)-y_j)^2 + \lam ||f||^2_\cH  \;,
\end{equation}
where $\lam >0$ is the {\it regularization parameter}. The Representer Theorem, see e.g. \cite{steinbook}, ensures that the solution $\hat f^\lam_\bz$ to \eqref{eq:min} 
exists, is unique and can be written as
\begin{equation}
\label{eq:estim}
  \hat f^\lam_\bz (x) = \sum_{j=1}^n  \alpha_j K(x_j , x) \, 
\end{equation}  
with
\[ \alpha = (\mbk_n + \lam n I)^{-1}\by \]
and where $\mbk_n = (K(x_i, x_j))_{i,j} \in \mbr^{n \times n}$ is the kernel matrix.  In particular, this means that minimization can be restricted to the space 
\[  \cH_n = \{ \;f \in \cH\;|\; f=\sum_j^n \alpha_j K(x_j, \cdot )\;, \alpha_j \in \mbr \; \} \;.\]

{\bf Rates of convergence and Optimality.} 
A common goal of learning theory is to give upper bounds for the convergence of $\hat f_\bz^{\lam_n}$ to $f_\rho$,  
where the regularization parameter is tuned according to sample size,  
and derive rates of convergence  as $n\rightarrow \infty$ under appropriate
assumptions on the regularity of $f_\rho$\,. 
In this paper, our bounds are given in the usual squared $L^2(\nu)$ distance with respect to the sampling distribution,
which is equal to the excess risk when using the squared loss, i.e. 
\begin{equation}
\label{eq:excess_risk}
\|\hat f^\lam-f_\rho\|^2_{L^2} = \cE(f^\lam) - \cE(f_\rho) \;.  
\end{equation}
More precisely, we are interested in bounding the averaged above error over the
draw of the training data (this is also called Mean Integrated Squared Error).

A common framework for expressing regularity of the target function is by means of the kernel covariance operator 
\[ T = \mbe[K_X \otimes K_X] \;. \] 
If there exists $r>0$ such that  
\begin{equation}
\label{eq:source}
  ||  T^{-r }f_\rho ||_\cH \leq R
\end{equation}
for some $R<\infty $, then $f_\rho$ is considered as regular. In particular, this assumption ensures that $f_\rho \in \cH$, see  e.g. \cite{FischStein17}.    
This type of regularity class, also called {\em source condition}, has been considered in a 
learning context by \cite{cusmale}, and \cite{optimalratesRLS} have established upper bounds for
the performance of KRLS over such classes. This has been extended to other
types of kernel regularization methods by \cite{CapYao10,DicFosHsu15,BlaMuc16}.

Furthermore, bounds on the generalization error 
also depend on the notion
of {\em effective dimension} of the data with respect to the regularization parameter $\lam$\,,
defined as
\begin{equation}
\label{def:eff_dim_nips}
 \cN( \lam) := \cN(T, \lam):=\mbox{Trace}[(T+\lam)^{-1}T]   \,. 
\end{equation}
An assumed bound of the form 
\[  \cN( \lam)  \lesssim \lam^{-\gamma} \]
with $0<\gamma \leq 1$ is referred  to as a {\it Capacity Assumption}, see \cite{Zha05}. 
In particular, it is shown in \cite{optimalratesRLS}  that \eqref{def:eff_dim_nips} is ensured if the 
eigenvalues\footnote{Note that boundednes of $K$ ensures that $T$ is trace class, hence compact and has a discrete spectrum.} $(\mu_j)_j$ of $T$ enjoy a polynomial decay, 
i.e. $\mu_j \lesssim j^{-\frac{1}{\nu}}$.  

It is well known and fairly standard that bounds of the excess risk \eqref{eq:excess_risk} are guided by the two conditions \eqref{eq:source} on regularity 
and \eqref{def:eff_dim_nips} on the capacity, i.e.
\begin{equation}
\label{eq:rate}
\mbe\brac{ \cE(f_\bz^{\lam_n}) - \cE(f_\rho)  } \lesssim  R^2 \paren{ \frac{1}{n}  }^{\frac{2r+1}{2r+1+\gamma}} \;,
\end{equation}
with $0<r\leq \frac{1}{2}$, provided the regularization parameter is chosen according to
\begin{equation}
\label{eq:lam_glob}
\lambda_n \simeq \paren{\frac{1}{n}}^{\frac{1}{2r+1+\gamma}} \;.
\end{equation} 
In the framework of KRLS, these bounds were derived in \cite{optimalratesRLS}; \cite{BlaMuc16} derive bounds in a more general framework. Both papers also show optimality 
(i.e. there is also a corresponding lower bound).

From \eqref{eq:rate} we immediately see that the regularity inherent in the problem has an impact on the speed of convergence: 
The larger the regularity, the faster is convergence.  


\section{Localization}
\label{sec:Partitioning_Approach}

In this section we introduce the partitioning approach and derive  our first main results.

\subsection{The Bottom Up Partitioning Approach}

We say that a family $\{\cX_1, ..., \cX_m\}$ of nonempty disjoint subsets of $\cX$   is a  {\it partition} of $\cX$, if $\cX=\bigcup _{j=1}^m\cX_j$.
Given a probability measure $\nu$ on $\cX$,
let $p_j= \nu(\cX_j)$. We endow each $\cX_j$ with a probability measure by restricting the conditional probability
$\nu_j(A):= \nu(A | \cX_j)=p_j^{-1}\nu (A \cap \cX_j)$ to the Borel sigma algebra on $\cX_j$. 


We further assume that $\cH_j$ is a (separable) RKHS, equipped with a measurable positive semi-definite real-valued kernel $K_j$ on each $\cX_j$, 
bounded by $\kappa_j$. Note that any function in $\cH_j$ is only defined on $X_j$. To make them globally defined, we extend each function $f \in \cH_j$ to a function 
$\hat f :\cX \longrightarrow \mbr$ by extending as the zero-function, i.e.  $\hat f(x) = f(x)$ for any $x \in \cX_j$ and $\hat f (x) = 0$ 
else. In particular, $\hat K_j$ denotes the  kernel extended  to $\cX$, explicitly given by $\hat K_j(x, x')=K_j(x, x')$ for any $x, x' \in \cX_j$ and zero else. 
Then the space $\hat \cH_j:=\{ \hat f: f \in \cH_j \}$ equipped with the norm $||\hat f||_{\hat \cH_j} = ||f||_{\cH_j}$ 
is again an RKHS of functions on $\cX$ with kernel $\hat K_j$. 
Finally, the direct sum
\[  \cH := \bigoplus_{j=1}^m \hat \cH_j = \{ \; \hat f=\sum_{j=1}^m \hat f_j \;:\; \hat f_j \in \hat \cH_j \;\}  \]
with  norm 
\[ ||\hat f||^2_{\cH}=\sum_{j=1}^m p_j \; || \hat f_j||^2_{\hat \cH_j} \] 
is also an RKHS for which
\begin{equation}
\label{def:kernel}
 K(x, x') = \sum_{j=1}^m p_j^{-1} \hat K_j(x, x')  \;, 
\end{equation} 
$ x, x' \in \cX$, is the reproducing kernel, see \cite{aron50}.


Given training data $\cD=\{x_i, y_i\}_{i \in [n]}$, 
we let 
\[ I_j=\{i \in [n]: x_i \in \cX_j\} \]
the set of indices indicating the samples associated to  $\cX_j$, with $|I_j|=n_j$. We split $\cD$ according to 
the above partition, i.e. we let $\cD_j=\{ x_i, y_i \}_{i \in I_j}$. We further let 
$\bx_j=(x_i)_{i \in I_j}$, $\by_j = (y_i)_{i \in I_j}$.  


Fixing a regularization parameter $\lam >0$, we compute for each $\cD_j$ a local KRLS estimator (compare with \eqref{eq:estim} in the global setting)
\begin{equation*}
\label{def:local_estimator}
\hat f_{\cD_j}^{\lam}:= \sum_{i \in I_j}  \alpha_j^{(i)} \hat K_j(x_i, \cdot) \in \hat \cH_j \; ,
\end{equation*}
where $\alpha_j \in \mbr^{n_j}$ is given by 
\[ \alpha_j = \paren{\mbk_j + n_j\lambda }^{-1}\by_j \]
and with $\mbk_j$ the kernel matrix associated to $\cD_j$. 

Finally, the overall estimator is defined by
\begin{equation}
\label{def:global_estimator}
\hat f^{\lam}_{\cD}:=\sum_{j=1}^m \hat f_{\cD_j}^{\lam} \;, 
\end{equation}
which  by construction belongs  $\cH$   and  decomposes according to the direct sum  $\hat \cH_1 \oplus ...\oplus \hat \cH_m$.



\subsection{Finite Sample Bounds}

Our aim is to give an upper bound for the expected excess risk 
\[ \mbe\brac{\; \cE(\hat f^{\lam}_{\cD}) - \cE(f_\rho) \;} \;. \] 
In view of the regularity assumptions made in the global setting and described 
in the previous section, it is now straightforward how to express local regularity: 


\begin{assumption}[Regularity]
\label{ass:zero}
\begin{enumerate}
\item  The regression function $f_\rho$ belongs to $\cH$ and thus has a unique representation $f_\rho = f_1+...+f_m$\;, with 
$f_j \in \hat \cH_j$\;.
\item
The local regularity of the regression function is measured in terms of a source condition:
\begin{equation} 
\label{ass:SC}
 ||T_j^{-r_j}f_j||_{\hat \cH_j} \leq R\;, \quad \quad 0<r_j\leq  \frac{1}{2}\;,
\end{equation}
with $R < \infty$. 
\end{enumerate}
\end{assumption}
Note that this Assumption implies a global regularity of $f_\rho$ as
\[ ||T^{-r}f||_{\cH} \leq R \;, \quad 0<r\leq  \frac{1}{2}\;,  \]
with $r = \min(r_1, ..., r_m)$ and $R < \infty$.

Furthermore, we need some compatibility between the local effective dimensions and the global effective dimension in 
terms of the local probabilities $p_j$\;.

\begin{assumption}[Capacity]
\label{ass:one}
\begin{enumerate}
\item The local effective dimensions obey 
\begin{equation} 
\label{Nnew}
m \sum_{j=1}^m p_j\;\cN( T_j, \lam) =\cO(\cN( T, m\lam))  \;.
\end{equation}
\item Global capacity: For some $0<\gamma \leq 1$ 
\[ \cN( T, \lam) \lesssim \lam^{-\gamma} \;. \]
\end{enumerate}
\end{assumption}

Eq. \eqref{Nnew} is in particular an exact equality if $p_j \equiv \frac{1}{m}$.


As in the global learning problem, the choice of the regularization parameter $\lam=\lam_n$ depending on the sample size $n$ is crucial for 
the algorithm to work well. Interestingly, our main result shows that choosing $\lam$ locally on each subset exactly in the same way as for the global learning KRLS problem 
(see \eqref{eq:lam_glob}) leads to the same error bounds as in \eqref{eq:rate}.

\begin{theo}[Finite Sample Bound]
\label{theo:fixed_partition}
Let $n_j =\lfloor \frac{n}{m}\rfloor$. 
Then, 
with the choice 
\begin{equation}
\label{def:lamchoice}
\lam_n \simeq   \paren{\frac{1}{n}}^{\frac{1}{2r+1+\gamma}}
\end{equation} 
and with 
\begin{equation}
\label{equation:n0}
  m \lesssim n^{\alpha} \;,\quad  \alpha \leq\frac{2r}{2r+1+\gamma}  \;. 
\end{equation}
we have the following error bound 
\begin{equation}
\label{eq:fixed_upper_part}
   \mbe\brac{\; \cE(\hat f^{\lam}_{\cD}) - \cE(f_\rho) \;}  \lesssim    R^2\;\left(\frac{1}{n} \right)^{\frac{2r+1}{2r+1+\gamma}}  \,.
\end{equation}
\end{theo}

Condition \eqref{equation:n0} tells us that the sample size needs to be large enough on each local set in order to guarantee meaningful bounds. We can see that 
{\it large enough} depends here on the regularity $r$ and the capacity $\gamma$. 

We emphasize that the rhs of \eqref{eq:fixed_upper_part} coincides (for the case $m=1$) with the minimax optimal rate of convergence, as shown in 
\cite{optimalratesRLS} and \cite{BlaMuc16}\;. 
Note that for $m>1$ there is no explicit proof of lower bounds available in the literature (because of our additional 
hypothesis \eqref{Nnew}\;, restricting the considered model class).




\subsection{Incorporating Locality: Improved Error Bounds}

Our result in Theorem \ref{theo:fixed_partition} shows that the error bound is indeed guided by the lowest degree of regularity. Next we show, that 
sometimes we can do even better if low regularity only occurs on a local set having small probability. To be more precise, 
assume that there is an exceptional set $E$
of indices such that the smoothness of $f_\rho$ is low on each set $\cX_j$, $j \in E$ and higher on each $\cX_j$, $j \in E^c$. 
For ease of reading we shall only analyze the most simple case given by:
\begin{assumption}[Regularity]
\label{ass:three}
There are $r_{l}, r_h \in (0,\frac{1}{2}]$, with $r_l < r_h$ (corresponding to {\it low} smoothness and {\it high} smoothness) and there are 
$R_l<\infty$, $R_h< \infty$ such that 
\[   ||T_j^{-r_l}f_j||_{\hat \cH_j} \leq R_l \;, \quad \forall j \in E \;, \]
\[   ||T_j^{-r_h}f_j||_{\hat \cH_j} \leq R_h \;, \quad \forall j \in E^c \;.\]
Furthermore, assume that for any $n$ sufficiently large 
$$\paren{\sum_{j \in E} p_j} \lesssim \paren{\frac{R_h}{R_l}}^2 \lam_n^{2(r_h - r_l)} \;.$$
Here, $\lam_n$ is given by \eqref{def:lamchoice2}.
\end{assumption}

Thus, global smoothness is given by the small degree $r_l$, while local smoothness on the complement of the exceptional set is higher.
We emphasize that this is an additional assumption on the sampling distribution $\nu$. 
Assumption \ref{ass:three} then ensures that the probability of the exceptional set is so small that the error bound will actually be governed 
by the higher smoothness $r_h$, leading to an improved finite sample bound. 
More precisely,

\begin{theo}[Improved error Bound]
\label{theo:fixed_partition_improved}
Let $n_j =\lfloor \frac{n}{m}\rfloor$. 
Then, 
with the choice 
\begin{equation}
\label{def:lamchoice2}
\lam_n \simeq  \paren{\frac{1}{n}}^{\frac{1}{2r_h+1+\gamma}}
\end{equation} 
and with 
\begin{equation}
\label{equation:n02}
  m \lesssim n^{\alpha} \;,\quad  \alpha\leq\frac{2r_h}{2r_h+1+\gamma}  \;. 
\end{equation}
we have the following improved error bound 
\begin{equation}
\label{eq:fixed_upper_part_improved}
  \mbe\brac{\; \cE(\hat f^{\lam}_{\cD}) - \cE(f_\rho) \;}  \lesssim  \;  R_h^2 \left(\frac{1}{n} \right)^{\frac{2r_h+1}{2r_h+1+\gamma}}  \,.
\end{equation}
\end{theo}
Again, for giving meaningful bounds 
the sample size needs to be large enough on each local set, depending on the regularity $r$ and capacity $\gamma$.




\section{KRLS Nystr\"om Subsampling}
\label{sec:KRR_Nys}

In this section we recall the popular KRLS Nystr\"om subsampling method.   
For simplicity, we restrict ourselves to so called {\em Plain Nystr\"om}, which works as follows:  Given a training set $x_1, ..., x_n$ of 
random inputs, we sample uniformly at random without replacement $l \leq n$ points $\tilde x_1, ..., \tilde x_l$. Now the crucial idea is to seek 
for an estimator for the unknown $f_\rho$ in a reduced space
\[  {\cal H}_l = \{ \;  f\;:\; f=\sum_{j=1}^l \alpha_j K(\tilde x_j,\cdot)\; , \;  \alpha \in \mbr^l \;  \} \,. \]
In \cite{RudCamRos15} it is shown that the solution of the minimization problem
\[ \min_{f \in {\cal H}_l } \; \frac{1}{n}\sum_{j=1}^n(f(x_j)-y_j)^2 + \lam ||f||^2_{{\cal H}_l} \]
is given by
\begin{equation}
\label{estimator:nystrom}
 \hat f_{n,l}^{\lam} = \sum_{j=1}^l \alpha_j K(\tilde x_j, \cdot) \;,
\end{equation}
with
\[  \alpha = \left(\mbk^*_{nl}\mbk_{nl} + n\lambda  \mbk_{ll}\right )^{\dagger}\;\mbk^*_{nl} \by \;, \]
where $(\mbk_{nl})_{ij}=K(x_i,\tilde x_j)$, $(\mbk_{ll})_{kj}=K(\tilde x_k,\tilde x_j)$, $i=1,...,n$\;,\; $k,j=1,...,l$ and 
$A^{\dagger}$ denotes the generalized inverse of a matrix $A$.

Clearly, one aims at minimizing the number $l$ of subsamples needed for preserving  minimax optimality. 
We amplify the results  in \cite{RudCamRos15} by explicitly 
computing how $l$ needs to grow when the total number of samples $n$ tends to infinity. 
We exhibit the explicit dependence on the regularity parameter $r$  and on the capacity assumption, parametrized by $\gamma$.
Furthermore, we refine the analysis in \cite{RudCamRos15} by deriving bounds in expectation  removing the dependence of $l$ on the confidence level. 
This will be crucial for deriving our optimality results in the next section. 

We consider the setting of Section \ref{sec:Partitioning_Approach} with $m=1$. Granted Assumptions \ref{ass:zero} and \ref{ass:one}, one has: 

\begin{theo}[KRLS-Plain Nystr\"om]
\label{theo:plain_nys}
If the number $l$ of subsampled points satisfies 
\begin{equation}
l\gtrsim n^{\beta} \;, \quad \beta \geq \frac{1+\gamma}{2r+1+\gamma}   \;,
\end{equation}
and if $r \in [0, \frac{1}{2}]$ then  the choice \eqref{def:lamchoice} for $(\lam_n)_n$ 
leads to the error bound 
\[
        \mbe\brac{  \cE(\hat f_{n, l_n}^{\lam_n}) -\cE(f_\rho ) }  \lesssim   R^2\paren{\frac{1}{n}}^{\frac{2r+1}{2r+1+\gamma}} \;. 
\]
\end{theo}

Our main result shows that the number of subsampled points can be substantially reduced from $\cO\paren{n^{\beta}\log(n)}$, see  \cite{RudCamRos15}, 
to actually $\cO\paren{n^{\beta}}$.


\section{Combining Localization and Subsampling}
\label{sec:localnysation}

In this section we establish, that upper rates of convergence are preserved if 
one combines the partitioning approach of Section \ref{sec:Partitioning_Approach} with the Nystr\"om subsampling approach of the previous section.
For simplicity we assume that the local sample size is roughly the same on each partition, i.e satisfies $n_j = \lfloor \frac{n}{m}\rfloor$ and that the 
number $l=l_n$ of subsample points also is equal on each subsample.

For $j =1,...,m$, and $1\leq l \leq \frac{n}{m}$ let $\tilde I_{j,l}:=\{i_{j,1}, ..., i_{j,l}\} \subseteq I_j$, with $I_j$ as above ($\tilde I_{j,l}$ 
denotes the set of indices of subsampled inputs on each $\cX_j$).   
For each subsample $\cD_j$, with a regularization parameter $\lam >0$, we compute a local estimator
\begin{equation*}
\label{def:local_estimator2}
\hat f_{\cD_j}^{\lam}:= \sum_{i \in \tilde I_{j,l}} \alpha_j^{(i)}(\lam) \hat K_j(x_i, \cdot) \in \hat \cH_{j,l} \; ,
\end{equation*}
where $\alpha_j \in \mbr^{\frac{n}{m}}$ is given in \eqref{estimator:nystrom}, with $n$ replaced by $\frac{n}{m}$. 
The overall estimator is constructed as above and defined by
\begin{equation}
\label{def:global_estimator2}
\hat f^{\lam}_{\cD}:=\sum_{j=1}^m \hat f_{\cD_j}^{\lam} \;, 
\end{equation}
which by construction  decomposes according to the direct sum $\cH=\hat \cH_{1} \oplus ...\oplus \hat \cH_{m}$\;. 
Then we have:

\begin{theo}
\label{theo:localnys}
Let $r = \min(r_1 ,..., r_m)$.  
If the number $l$ of subsampled points on each local set satisfies 
\begin{equation}
l \sim n^\beta \;, \quad \beta =  \frac{1+\gamma}{2r+1+\gamma}\;,
\end{equation}
and if the number of local sets satisfies
\begin{equation*}
  m \lesssim n^{\alpha} \;,\quad  \alpha \leq \frac{2r}{2r+1+\gamma}  \;, 
\end{equation*}
then  the choice \eqref{def:lamchoice} for the regularization parameter $\lam_n$ guarantees the error bound 
\begin{equation}
\label{eq:upper_part2}
      \mbe\brac{ \cE ( \hat f_{\cD}^{\lam_{n}} ) - \cE(f_\rho)}  \lesssim  R^2\; \paren{\frac{1}{n}}^{\frac{2r+1}{2r+1+\gamma}}  \,,
\end{equation}
provided $n$ is sufficiently large.
\end{theo}

Clearly, as in Theorem \ref{theo:fixed_partition_improved}, a version of the above result still holds if global smoothness is violated on an exceptional set 
$E$ of small probability as amplified in Assumption \ref{ass:three}. 
We leave a precise formulation (and its proof) to the reader.


\section{Discussion and Comparison to other Approaches}
\label{sec:distributed}
First results establishing learning rates using a KRLS partition-based approach for smoothness parameter $r=0$ and polynomially decaying eigenvalues 
are given in \cite{Tandon16}. The authors establish  
upper  rates of convergence under an additional assumption on the probability of the local sets $\cX_j$, 
requiring the existence of sufficiently high moments in $L^2(\nu)$ of the eigenfunctions of their local covariance operators, uniformly 
over all subsets, in the limit $n \to \infty$. 
However,  
while the decay rate of the eigenvalues can be determined by the smoothness  of  $K$ (see e.g. \cite{Men09} and references therein) 
it is a widely open question  which (general) properties of the kernel imply such assumptions on the eigenfunctions. 
We remove these assumptions on the eigenfunctions of the covariance operator which are restrictive and difficult to prove. In addition,  
we allow locally different degrees of smoothness, improving finite sample bounds.

The paper {\cite{EbStein15}} considers localized SVMs, localized tuned Gaussian kernels and a corresponding direct sum decomposition, 
where a global smoothness assumption is introduced in terms of a scale of Besov spaces. Instead of using the effective dimension $\cN(\lam)$ as a measure for complexity,  
the authors use entropy numbers, obtaining minimax optimal rates.  We extend these results by  going beyond Gaussian 
kernels and allowing more general input spaces than open subsets of $\mbr^d$, allowing in addition the choice of different local kernels.

We also compare the partitioning approach with distributed learning (parallelizing) for KRLS, as recently 
analyzed in \cite{guo17} and \cite{MuBla18}. The distributed learning algorithm is based on a uniform partition of the given 
data set  $$\cD=\{(X_1, Y_1), ..., (X_n, Y_n)\} \subset \cX \times \mbr$$ into $m$ disjoint equal-size subsets $\cD_1, ..., \cD_m$. 
On each subset $\cD_j$, one computes a local estimator $\hat f^{\lam}_{D_j}$ using KRLS (or more general, a spectral regularization method). 
The  final estimator is given by simple averaging: $\bar f^{\lam}_{D}: = \frac{1}{m}\sum_{j=1}^m \hat f^{\lam}_{D_j}$. 

In this setting, one takes a similiar point of view as in our main Theorem \ref{theo:fixed_partition}.  
Both, \cite{guo17} and \cite{MuBla18} provide an answer to the question: How much is the number $m$ of local machines allowed to grow with the sample size $n$ 
in order to preserve minimax optimal rates of convergence? It has been shown by these authors, that 
\[  m_n \sim n^{\alpha}\;, \qquad \alpha\leq \frac{2r}{2r+1+\gamma} \]
gives a sufficient condition.  Here, $r \in (0,\frac{1}{2}]$ is again the regularity parameter of the objective function and $0<\gamma \leq 1$ 
characterizes the decay of 
the effective dimension. Note that this relation between sample size $n$ and number $m$ of subsamples precisely agrees with our equation \eqref{equation:n0}\;. 
We have condensed the computational cost of all these methods in Table \ref{mytable}.

\begin{table}[t]
  \caption{Computational Cost}
  \label{mytable}
  \centering
  \begin{tabular}{ll}
    \hline
    KRLS & ${\cal O}(n^3)$ \\
    localized KRLS & ${\cal O}\paren{(\frac{n}{m})^3 }$, \;$1\leq m \leq n^{\alpha}$\\
    Nystr\"om & ${\cal O}(nl^2+l^3) $,\; $n^{\beta} \leq l \leq n $\\
    local Nys. &${\cal O}(\frac{n}{m}l^2 + l^3) $, \;$n^{\beta}\leq l \leq \frac{n}{m}$ \\
    distributed KRLS & ${\cal O}\paren{(\frac{n}{m})^3 }$, \;$1\leq m \leq n^{\alpha}$\\
    \hline
  \end{tabular}
\end{table}




\section{Conclusion}
\label{sec:conclusion}

We have shown that the twofold effect of partitioning and subsampling  may substantially reduce computational cost, 
if the number of local sets is sufficiently small w.r.t. the amount of data at hand and if the number of subsampled inputs 
is sufficiently large w.r.t. the sample size. 
In both cases we were able to improve 
or amplify the existing results. Furthermore, we derived a rigorous version of the principle 
{\em In partitioning, low smoothness on exceptional sets of small probability does not affect finite sample bounds}. 

\vspace{2cm}

\noindent
{\Large \bf Acknowledgments}
\\
\\
The author acknowledges support by the German Research Foundation under  DFG Grant STE 1074/4-1. 
Furthermore, the author is grateful to  Markus Klein for useful discussions.

\bibliographystyle{plain}
\bibliography{bibliography}

\appendix


\section{Preliminaries}

We let $\cZ = \cX \times \mbr$ denote the sample space, where the input space $\cX$ is a standard Borel space endowed with a fixed unknown probability measure $\nu$.   
The kernel space $\cH$ is assumed to be separable, equipped with a measurable positive semi-definite kernel $K$, bounded by $\kappa$, implying 
continuity of the inclusion map $I : \cH \longrightarrow  L^2(\nu)$. Moreover, we consider the covariance operator 
$T=\kappa^{-2}I^*I=\kappa^{-2}\mbe [ K_X \otimes K_X ] $, which can be shown to be positive self-adjoint trace class (and hence is compact). 
Given a sample $\bx=(x_1, \ldots, x_n) \in \cX^n$, we define the sampling operator $S_{\bx}: \cH \longrightarrow  \mbr^n$ by 
$(S_{\bx}f)_i= \inner{ f, K_{x_i}}_{\cH}$. The empirical 
covariance operator is given by $T_{\bx}=\kappa^{-2}S_{\bx}^*S_{\bx}: \cH \longrightarrow  \cH$. 

For a partition $\{\cX_1, ... , \cX_m \}$ of $\cX$, we denote by $\hat \cH_j$ the local RKHS with extended bounded kernel $\hat K_j$, supported on $\cX_j$, 
with associated covariance 
operator $T_j = \kappa_j^{-2}\mbe_{\nu_j} [ \hat K_j(X, \cdot) \otimes \hat K_j(X, \cdot) ]$. 
Given a sample $\bx_j=(x_{j,1}, \ldots, x_{j,n_j}) \in \cX_j^{n_j}$, we define the sampling operator $S_{\bx_j}: \hat \cH_j \longrightarrow  \mbr^{n_j}$ 
similarly by $(S_{\bx_j}f)_i= \inner{ f, \hat K_j(x_i, \cdot) }_{\hat \cH_j}$.

The global covariance operator acts as an operator on the direct sum $\cH = \hat \cH_1 \oplus ... \oplus \hat \cH_m$. According to \eqref{def:kernel}, it decomposes as
\[ T= \sum_{j=1}^m p_j^{-1} T_j \;,  \]
which can be used to prove that the global effective dimension can be expressed as the sum of the (rescaled) local ones. 

\begin{lemma}[Effective Dimension]
\label{lem:eff_dim_sum}
For any $\lam \in [0,1]$
\[  \sum_{j=1}^m \cN( T_j,p_j \lam  ) = \cN ( T , \lam  ) \;. \]
\end{lemma}

Finally, our error decomposition relies on the the following standard decomposition

\begin{lemma}
\label{lem:L2norm}
Given $j \in [m]$ let $p_j=\nu(\cX_j)$ and $\nu_j(A) = \nu(A | \cX_j)$, for a measurable $A \subset \cX$. One has 
\[   L^2(\cX, \nu) = \bigoplus_{j=1}^m p_j L^2(\cX_j, \nu_j)  \]
with  
\[ ||f||^2_{L^2(\nu)} = \sum_{j=1}^m p_j|| f_j||^2_{L^2(\nu_j)}\;, \]
where $f=f_1+...+f_m$\;. 
\end{lemma}

For proving our results we additionally need an appropriate Bernstein condition on the noise. 

\begin{assumption}[Distributions]
\label{ass:two}
\begin{enumerate}
\item The sampling is 
random i.i.d., where each observation point $(X_i,Y_i)$ follows the model
$ Y = f_\rho(X) + \epsilon \,,$ 
and the noise satisfies the following Bernstein-type assumption:
For any integer $k \geq 2$ 
and some $\sigma > 0$ and $M>0$: 
\begin{equation}
\label{def:bernstein}
 \mbe[\; |Y-f_\rho(X)|^{k} \; | \; X \;] \leq \frac{1}{2}k! \; \sigma^2 M^{k-2} \quad \nu - {\rm a.s.} \;.  \tag{Bern($M$,$\sigma$)}
\end{equation}
\item
Given $\theta=(M, \sigma, R) \in \mbr^3_+$, the class $\cM :=\cM(\theta , r, b)$ consists of all distributions $\rho$ with 
$X$-marginal $\nu$ and conditional distribution of \;$Y$ given $X$
satisfying \eqref{def:bernstein}  for the deviations and \eqref{ass:SC} for the mean.
\end{enumerate}
\end{assumption}

We remark  that point 1 implies for any $j \in [m]$
\begin{equation}
 \mbe[\; |Y-f_j(X) |^{k} \; | \; X \;] \leq \frac{1}{2}k! \; \sigma^2 M^{k-2} \quad \nu_j - {\rm a.s.},   
\end{equation}
where $\sigma$ and $M$ are uniform with respect to $m$ and $k$. This is what we actually need in our proofs.

\[\]
\noindent
For ease of reading we make use of the following conventions: 
\begin{itemize}
 \item we are interested in a precise dependence of multiplicative constants on the parameter $\sigma, M, R$, $m$, $n$ 
 \item the dependence of multiplicative constants on various other parameters, including the kernel parameter $\kappa$, the parameters arising from the regularization method, $b>1$, $r>0$,  etc.  will
   (generally) be omitted 
 \item the value of $C$ might change from line to line  
 \item the expression ``for $n$ sufficiently large'' means that the statement holds for
   $n\geq n_0$\,, with $n_0$ potentially depending on all model parameters
   (including $\sigma, M$ and $R$)\,.
 \end{itemize}


\section{Proofs of Section \ref{sec:Partitioning_Approach}}

This section is devoted to proving the results of Section \ref{sec:Partitioning_Approach}. Recall that by Assumption \ref{ass:zero} the regression 
function belongs to $\cH$, i.e. admits an unique representation $f=f_1 +...+f_m$\;, with $f_j \in \hat \cH_j$. For proving our error bounds we shall 
use a classical bias-variance decomposition

\begin{equation*}
 f_\rho - \hat f^{\lam}_{\cD}  = \sum_{j=1}^m  f_j- \hat f^{\lam}_{\cD_j} 
 = \sum_{j=1}^m  r_{\lam}( T_{\bx_j}) f_j  + 
     \sum_{j=1}^mg_{\lam}( T_{\bx_j})( T_{\bx_j} f_j -  S^*_{\bx_j}\by_j) \;,
\end{equation*}
where $\hat f^{\lam}_{\cD}$ is given in \eqref{def:global_estimator}, with $r_{\lam}(t)= 1-g_{\lam}(t)t$ and with $g_{\lam}(t)=(t+\lam)^{-1}$. 
The final error bound follows then from 
\begin{align}
\label{eq:final}
 \mbe\brac{\; \cE(f_\rho) - \cE(\hat f^{\lam}_{\cD})  \;}  &= \mbe\brac{\; ||f_\rho - \hat f^{\lam}_{\cD} ||^2_{L^2(\nu)}  \;} \nonumber \\
&\leq \mbe\brac{\; ||\sum_{j=1}^m  r_{\lam}( T_{\bx_j}) f_j||^2_{L^2(\nu)} \;}  +  
\mbe\brac{\;||\sum_{j=1}^mg_{\lam}( T_{\bx_j})( T_{\bx_j} f_j -  S^*_{\bx_j}\by_j)||^2_{L^2(\nu)}\;}   \;.
\end{align}

We proceed by bounding each term in the above decomposition separately.

\begin{prop}[Approximation Error]
\label{prop:approx_error_part}
For any $\lam \in (0,1]$, one has 
\begin{align*}
\mbe\Big[ \big\|  \sum_{j=1}^m r_{\lam}( T_{\bx_j})f_j  \big\|_{L^2(\nu)}^2  \Big] 
 &\leq  C R^2\sum_{j=1}^m p_j  \cB^2_{n_j}( T_j ,\lam) \lam^{2(r_j+\frac{1}{2})} \;,
\end{align*}
where $\cB^2_{n_j}( T_j ,\lam) $ is defined in Proposition \ref{prop:Guo_nips} and where $ C$ does not depend on $(\sigma, M,R)\in \mbr^3_+$\;.
\end{prop}

\begin{proof}[Proof of Proposition \ref{prop:approx_error_part}]
Recall\footnote{If $I_j: \hat \cH_j \hookrightarrow L^2(\nu_j)$, then $T_j=I_j^*I_j$ and 
$||\sqrt{ T_j}  f||^2_{\hat \cH_j} = \inner{T_j f , f}_{\hat \cH_j}=\inner{I_j f, I_j f}_{L^2(\nu_j)} = ||f||^2_{L^2(\nu_j)}$. Here, we 
identify $I_j f= f$.} that $||\sqrt{ T_j}  f||_{\hat \cH_j} = || f||_{L^2(\nu_j)}$ for any $f \in \hat \cH_j$. According to Lemma \ref{lem:L2norm}, by Assumption \ref{ass:SC} we have
\begin{align}
\label{est:first}
\mbe\Big[ \big\|  \sum_{j=1}^m r_{\lam}(T_{\bx_j}) f_j  \big\|_{L^2(\nu)}^2  \Big] 
 &= \sum_{j=1}^m p_j\mbe\Big[ \big\|r_{\lam}( T_{\bx_j}) f_j   \big\|_{L^2(\nu_j)}^2  \Big] \nonumber \\  
  &= \sum_{j=1}^m p_j\mbe\Big[ \big\| \sqrt{ T_j} r_{\lam}( T_{\bx_j})f_j   \big\|_{\hat \cH_j}^2  \Big]  \nonumber\\  
 &\leq  C R^2\sum_{j=1}^m p_j \mbe\Big[ \big\|\sqrt{ T_j}r_{\lam}( T_{\bx_j})  T_j^{r_j}   \big\|^2  \Big] \;.
\end{align}

We bound for any $j \in [m]$ the expectation by first deriving a probabilistic estimate.   
For any $\eta \in (0,1]$, with probability at least $1-\eta$
\begin{small}
\begin{align*}
||\sqrt{ T_j}r_{\lam}( T_{\bx_j})  T_j^{r_j} || &\leq C \log^2(2\etainv)\cB_{n_j}( T_j ,\lam)\; ||  T_j^{\frac{1}{2}}(  T_j + \lam)^{\frac{1}{2}} || \; 
||( T_{\bx_j}+\lam)^{\frac{1}{2}}r_{\lam}( T_{\bx_j})( T_{\bx_j} + \lam)^{r_j}|| 
              \; ||(  T_j + \lam)^{r_j} T_j^{r_j} || \\
              &\leq  C \log^2(2\etainv)   \cB_{n_j}( T_j ,\lam) \lam^{r_j+\frac{1}{2}} \;.
\end{align*}
\end{small}
Here we have used that 
\[  ||( T_{\bx_j}+\lam)^{\frac{1}{2}}r_{\lam}( T_{\bx_j})(T_{\bx_j} + \lam)^{r}|| \leq C \lam^{r_j+\frac{1}{2}} \]
and that for $s \in [0,\frac{1}{2} ]$
\[  ||(  T_j + \lam)^{s} T_j^{s} || \leq  ||(  T_j + \lam)T_j ||^{s} \leq 1 \]
by Proposition \ref{prop:Cordes} and the spectral theorem. Also, from Proposition \ref{prop:Cordes} and Proposition \ref{prop:Guo_nips}
\[  || ( T_{\bx_j} + \lam)^{-\frac{1}{2}}(T_{j} + \lam)^{\frac{1}{2}}  || 
\leq ||( T_{\bx_j} + \lam)^{-1}( T_{j} + \lam)||^{\frac{1}{2}} \leq \sqrt 8 \log(2\etainv)  \cB^{\frac{1}{2}}_{n_j}( T_j ,\lam)  \;.\]
From Lemma \ref{lem:integration_nips}, by integration
\begin{align*}
\mbe\Big[ \big\|\sqrt{ T_j}r_{\lam}(T_{\bx_j}) T_j^{r_j}   \big\|^2  \Big]  
 &\leq C \cB^2_{n_j}( T_j ,\lam) \lam^{2(r_j+\frac{1}{2})} \;.
\end{align*}
Combining this with \eqref{est:first} finishes the proof. 
\end{proof}

\begin{prop}[Sample Error]
\label{prop:sample_error}
For any $\lam \in (0,1]$, one has 
\[ \mbe
     \Big[\big\|\sum_{j=1}^m g_{\lam}( T_{\bx_j})( T_{\bx_j} f_j -  S^*_{\bx_j}\by_j) \big\|^2_{L^2(\nu)}\Big] \leq C 
     \sum_{j=1}^m p_j\;
           \cB_{n_j}^2( T_j ,\lam)\lam\;  \left( \frac{M}{n_j \lambda} + \sigma\sqrt{\frac{ {{\cal N}( T_j,\lam)}}{{n_j \lam }}}\right)^2 \;,\]
where $\cB^2_{n_j}( T_j ,\lam) $ is defined in Proposition \ref{prop:Guo_nips} and $ C$ does not depend on $(\sigma, M,R)\in \mbr^3_+$\;.
\end{prop}

\begin{proof}[Proof of Proposition \ref{prop:sample_error}]
Using again $||\sqrt{ T_j}  f||_{\hat \cH_j} = || f||_{L^2(\nu_j)}$ we find with Lemma \ref{lem:L2norm} 
\begin{align}
\label{est:first_2}
\mbe\Big[ \big\|  \sum_{j=1}^m g_{\lam}( T_{\bx_j})( T_{\bx_j} f_j -  S^*_{\bx_j}\by_j)  \big\|_{L^2(\nu)}^2  \Big] 
 &= \sum_{j=1}^m p_j\mbe\Big[ \big\|g_{\lam}( T_{\bx_j})( T_{\bx_j} f_j -  S^*_{\bx_j}\by_j) \big\|_{L^2(\nu_j)}^2  \Big] \nonumber \\  
  &= \sum_{j=1}^m p_j\mbe\Big[ \big\| \sqrt{ T_j} g_{\lam}( T_{\bx_j})( T_{\bx_j} f_j -  S^*_{\bx_j}\by_j)  \big\|_{\hat \cH_j}^2  \Big]  \;.  
\end{align}
We bound the expectation for each separate subsample of size $n_j$ by first deriving a probabilistic estimate and then by integration. 
For this reason, we use \eqref{Guo:estimate_nips} and Proposition \ref{prop:Cordes} and write for any $ f_j \in \hat \cH_j$,  $j \in [m]$
\begin{align}
\label{eq:needed}
   || \sqrt{ T_j} f_j||_{\hat \cH_j} & \leq || \sqrt{ T_j}(T_j+\lam)^{-1/2} ||\; || (T_j+\lam)^{1/2}(T_{\bx_j}+\lam)^{-1/2}   || \; 
                                ||(T_{\bx_j} +\lam)^{1/2} f_j  ||_{\hat \cH_j}  \nonumber \\
   &\leq ||  T_j(T_j+\lam)^{-1} ||^{1/2} \; || (T_j+\lam)(T_{\bx_j}+\lam)^{-1}   ||^{1/2}\; ||( T_{\bx_j} +\lam)^{1/2} f_j  ||_{\hat \cH_j}
     \nonumber \\ 
   &\leq   C\log(4\etainv)   \cB_{n_j}^{1/2}( T_j ,\lam)  \; ||( T_{\bx_j} +\lam)^{1/2} f_j  ||_{\hat \cH_j} \;,
\end{align}   
holding with probability at least $1-\frac{\eta}{2}$.

We proceed by splitting
\begin{equation}
\label{eq:splitting}
 ( T_{\bx_j} + \lam)^s g_{\lam}( T_{\bx_j})( T_{\bx_j}\fo - S_{\bx_j}^*\by_j )  =  H_{\bx_j}^{(1)}\cdot H_{\bx_j}^{(2)} \cdot h^{\lam}_{\bz_j} \;, 
\end{equation}
with
\begin{eqnarray*}
H_{\bx_j}^{(1)} &:=& ( T_{\bx_j} + \lam)^{\frac{1}{2}} g_{\lam}( T_{\bx_j})(T_{\bx_j} + \lam)^{\frac{1}{2}} ,\\
H_{\bx_j}^{(2)} &:=& ( T_{\bx_j} + \lam)^{-\frac{1}{2}}( T + \lam)^{\frac{1}{2}} , \\
h^{\lam}_{\bz_j} &:=& ( T + \lam)^{-\frac{1}{2}} ( T_{\bx_j}\fo -  S_{\bx_j}^*\by_j ) \;.
\end{eqnarray*}
The first term is bounded.  The second term is now estimated using \eqref{Guo:estimate_nips} once more. One has with probability at least $1-\frac{\eta}{4}$
\begin{equation*}
 H_{\bx_j}^{(2)} \leq \sqrt 8\log(8\etainv) \cB_{\frac{n}{m}}( T_j ,\lam)^{\frac{1}{2}} \;. 
\end{equation*} 
Finally, $h^{\lam}_{\bz_j}$ is estimated using Proposition \ref{Geta2_nips}:
\begin{equation*}
 h^{\lam}_{\bz_j} \leq 2\log(8\etainv)  \left( \frac{M}{n_j\sqrt{\lambda}} + \sigma\sqrt{\frac{ {{\cal N}( T_j, \lam)}}{{n_j}}}\right)\,, 
\end{equation*}
holding with probability at least $1-\frac{\eta}{4}$. Thus, combining the estimates following \eqref{eq:splitting} with \eqref{eq:needed} gives 
for any $j\in [m]$
\[ || \sqrt{ T_j} g_{\lam}( T_{\bx_j})( T_{\bx_j}\fo-  S^{*}_{ \bx_j }\by_j ) ||_{\hat \cH_j} \leq  
   C \log^{3}(8\etainv) \cB_{n_j}( T_j ,\lam)\sqrt{\lam}\;  \left( \frac{M}{n_j \lambda} + \sigma\sqrt{\frac{ {{\cal N}( T_j, \lam)}}{{n_j \lam }}}\right)\;,\]
with probability at least $1-\eta$. By integration using Lemma \ref{lem:integration_nips} one obtains 
\[  \mbe
     \Big[\big\|\sqrt{  T_j} g_{\lam}( B_{\bx_j})( T_{\bx_j}\fo-  S^{*}_{ \bx_j }\by_j ) \big\|^2_{\hat \cH_j}\Big]^{\frac{1}{2}} \leq  
     C \cB_{n_j}( T_j ,\lam)\sqrt \lam \;  \left( \frac{M}{n_j \lambda} + \sigma\sqrt{\frac{ {{\cal N}( T_j, \lam)}}{{n_j \lam }}}\right) \;.\] 
Combining this with \eqref{est:first_2} implies
\[ \mbe
     \Big[\big\|\sum_{j=1}^m g_{\lam}( T_{\bx_j})( T_{\bx_j}f_j -  S^*_{\bx_j}\by_j) \big\|^2_{L^2(\nu)}\Big] \leq C 
     \sum_{j=1}^m p_j\;
           \cB_{n_j}^2( T_j ,\lam)\lam\;  \left( \frac{M}{n_j \lambda} + \sigma\sqrt{\frac{ {{\cal N}( T_j, \lam)}}{{n_j \lam }}}\right)^2 \;,\]
where $ C$ does not depend on $(\sigma, M,R)\in \mbr^3_+$.
\end{proof}

\noindent
We are now ready to prove Theorem \ref{theo:fixed_partition}. 

\begin{proof}[Proof of Theorem \ref{theo:fixed_partition}]
Let the regularization parameter $\lam_n$ be chosen as
\begin{equation}
\lam_n = \min\paren{1, \paren{\frac{\sigma^2}{R^2n}}^{\frac{1}{2r+1+\gamma}} }\;,
\end{equation}
with $r=\min(r_1, ..., r_m)$ and assume that $n_j =\lfloor \frac{n}{m}\rfloor$. 
Note that by Lemma \ref{lem:fixed1} we have $\cB_{\frac{n}{m}}( T_j, \lam_n ) \leq 2$ for any $j \in [m]$, provided $n>n_0$, with $n_0$ 
given by \eqref{def:n0}. Since $\lam_n^{r_j} \leq \lam_n^{r}$ for any $j \in [m]$, the approximation error bound becomes by Proposition \ref{prop:approx_error_part} 
\begin{align}
\label{est:approxi}
\mbe\Big[ \big\|  \sum_{j=1}^m r_{\lam_n}( T_{\bx_j})f_j  \big\|_{L^2(\nu)}^2  \Big] 
 &\leq  C R^2\sum_{j=1}^m p_j   \lam_n^{2(r_j+\frac{1}{2})} \nonumber \\
 &\leq C R^2 \; \lam_n^{2(r+\frac{1}{2})}\;,
\end{align}
where we also used that $\sum_j p_j = 1$. 

For estimating the sample error firstly observe that 
\[  \frac{Mm}{n\lam_n} \leq R\lam_n^r \]
if
\[  n > \paren{m\;\frac{M}{R}}^{\frac{2r+1+\gamma}{r+\gamma}} \paren{\frac{R}{\sigma}}^{\frac{2(r+1)}{r+\gamma}} =: n_1\;. \]
Thus, from Proposition \ref{prop:sample_error} we obtain (recalling again that $\cB_{\frac{n}{m}}( T_j, \lam_n ) \leq 2$)
\begin{align}
\label{est:sample}
\mbe
     \Big[\big\|\sum_{j=1}^m g_{\lam_n}( T_{\bx_j})( T_{\bx_j} f_j -  S^*_{\bx_j}\by_j) \big\|^2_{L^2(\nu)}\Big] 
&\leq C \lam_n \sum_{j=1}^m p_j\paren{ R\lam_n^{r} + \sigma \sqrt{\frac{m\cN(T_j, \lam_n)}{n\lam_n}}}^2 \;.
\end{align}
We proceed by applying $(a+b)^2 \leq 2(a^2+b^2)$. Observe that by our Assumption \ref{ass:one}\;, $2.$ 
\begin{align}
\label{est:step}
 \sum_{j=1}^m p_j \sigma^2 \frac{m\cN(T_j, \lam_n)}{n\lam_n} &=  \sigma^2 \frac{m}{n\lam_n}  \sum_{j=1}^m p_j \; \cN(T_j, \lam_n) \nonumber  \\
 &\leq C  \frac{\sigma^2}{n\lam_n} \cN( T , m\lam_n) \nonumber \\
 &\leq C m^{-\gamma}\frac{\sigma^2}{n\lam_n} \lam_n^{-\gamma} \nonumber \\
 &\leq C R\lam_n^r \;,
\end{align}
by definition of $\lam_n$. Finally, combining \eqref{eq:final} with \eqref{est:step}, \eqref{est:sample} and \eqref{est:approxi} proves the theorem, provided 
\begin{equation}
\label{def:n01}
n>\max(n_0, n_1) \geq C_{M, \sigma, R,\gamma,r} \; m^{1+\frac{\gamma+1}{2r}} \;,
\end{equation}
for some (explicitly given) $C_{M, \sigma, R,\gamma,r} < \infty$.
\end{proof}


\[\]

\begin{proof}[Proof of Theorem \ref{theo:fixed_partition_improved}]
Assume that $n_j = \lfloor \frac{n}{m}\rfloor$. Let the regularization parameter $\lam_n$ be given by \eqref{def:lamchoice2}\;. 
As above, Lemma \ref{lem:fixed1} yields $\cB_{\frac{n}{m}}(T_j,\lam_n) \leq 2$ provided $n>n_0$, with $n_0$ satisfying \eqref{def:n0} (with $r$ replaced by $r_h$). 
From Proposition \ref{prop:approx_error_part} we immediately obtain for the approximation error
\begin{small}
\begin{align*}
\mbe\Big[ \big\|  \sum_{j=1}^{m} r_{\lam_n}( T_{\bx_j}) f_j  \big\|_{L^2(\nu)}^2  \Big] 
 &\leq  C \paren{R_l^2 \paren{\sum_{j \in E} p_j } \lam_n^{2(r_l+\frac{1}{2})} + R_h^2\paren{\sum_{j \in E^c} p_j } \lam_n^{2(r_h+\frac{1}{2})}  }\\
 &\leq CR_h^2\lam_n^{2(r_h+\frac{1}{2})} \;.
\end{align*}
Here we have used that by Assumption \ref{ass:three} 
\[  \paren{\sum_{j \in E} p_j} \leq \paren{\frac{R_h}{R_l}}^2 \lam_n^{2(r_h - r_l)} \quad \mbox{and}\quad \paren{\sum_{j \in E^c} p_j } \leq 1 \;. \]
\end{small}
The bound for the sample error follows exactly as in the proof of Theorem \ref{theo:fixed_partition}. Finally, the error bound \eqref{eq:fixed_upper_part_improved} 
is obtained by using again \eqref{eq:final}. 
\end{proof}


\section{Proofs of Section \ref{sec:KRR_Nys}}

For proving Theorem \ref{theo:plain_nys} we use the non-asymptotic error decomposition given in  Theorem 2 of \cite{RudCamRos15}, 
somewhat reformulated and streamlined using our estimate \eqref{Guo:estimate_nips}. 
We adopt the notation and idea of \cite{RudCamRos15} and write $\hat f^{\lam}_{n,l} = g_{\lam , l}( T_{\bx})S^*_{\bx}\by$, with 
$g_{\lam , l}(T_{\bx}) = V(V^* T_{\bx}V + \lam)^{-1}V^*$ and $VV^*=P_l,$ the projection 
operator onto $\cH_l$, $l \leq n$.  
Consider
\begin{align*}
|| \sqrt{ T}( \hat f^{\lam}_{n,l} - f_\rho) ||_{\cH} &\leq T_1 + T_2 \;
\end{align*} 
with 
\begin{small}
\[
T_1 =||  g_{\lam , l}(T_{\bx})(S^*_{\bx}\by - T_{\bx} f_\rho) ||_{L^2(\nu)} = || \sqrt{ T} g_{\lam , l}(T_{\bx})(S^*_{\bx}\by - T_{\bx} f_\rho) ||_{\cH}
 \]
and 
\[ T_2 = || \sqrt{ T} g_{\lam , l}(T_{\bx})(T_{\bx} f_\rho - f_\rho )||_{\cH} \;,
\]
\end{small}
which we bound in Proposition \ref{prop:sample_nys} and Proposition \ref{prop:approx_nys}\;.

\begin{prop}[Expectation Sample Error KRLS-Nystr\"om]
\label{prop:sample_nys}
\[ 
\mbe \Big[\big\|    g_{\lam , l}(T_{\bx})(S^*_{\bx}\by - T_{\bx} f_\rho)   \big\|_{L^2(\nu)}^2\Big]^{\frac{1}{2}} 
        \leq   C \; \sqrt{\lam} \cB_n( T, \lam)\paren{  \frac{M}{n\lam}  + \sigma \sqrt{\frac{\cN( T, \lam)}{n\lam}}}
 \]
 where $C $ does not depend on $(\sigma , M, R) \in \mbr_+^3$.
\end{prop}

\begin{proof}[Proof of Proposition  {\ref{prop:sample_nys}}]

For estimating $T_1$ we use Proposition \ref{prop:Guo_nips} and obtain for any $\lam \in (0, 1]$  with probability at least $1-\eta$
\begin{align*}
T_1 &\leq C\log (2\etainv )\cB_n(T, \lam)\; || ( T_{\bx} + \lam )^{1/2} g_{\lam , l}(T_{\bx}) (S^*_{\bx}\by -  T_{\bx}f_\rho) ||_{\cH} \\
&\leq C \log^{2} (4\etainv )\cB^2_n( T, \lam) \;|| ( T_{\bx} + \lam )^{1/2} g_{\lam , l}(T_{\bx}) ( T_{\bx} + \lam )^{1/2}|| \\
   & \quad  \quad ||( T + \lam )^{-1/2}(S^*_{\bx}\by -  T_{\bx}f_\rho)  ||_{\cH} \;. \\
\end{align*} 
From Proposition 6 in  \cite{RudCamRos15} and from the spectral Theorem we obtain
\[  || (T_{\bx} + \lam )^{1/2} g_{\lam , l}(T_{\bx}) ( T_{\bx} + \lam )^{1/2}|| \leq 1 \;. \]
Thus, applying  Proposition $\ref{Geta1_nips}$ one has with probability at least $1-\eta$
\[ T_1 \leq  C \log^{3} (8\etainv )\; \sqrt{\lam}\;\cB^2_n( T, \lam)\; \paren{  \frac{M}{n\lam}  + \sigma \sqrt{\frac{\cN( T,\lam)}{n\lam}}}\;, \]
where $C$ does not depend on $(\sigma , M, R) \in \mbr_+^3$. Integration using Lemma \ref{lem:integration_nips} gives the result. 
\end{proof}

Before we proceed we introduce the {\emph computational error}: For $u \in [0, \frac{1}{2}]$, $\lam \in (0,1]$ define
\[ {\cal C}_u(l, \lam):= ||(Id-VV^*)( T+\lam )^u||  \;.\]

The proof of the following Lemma can be found in \cite{RudCamRos15}, Proof of Theorem 2. 

\begin{lemma}
\label{lem:comp_error}
For any $u \in [0, \frac{1}{2}]$
\[ {\cal C}_u(l, \lam) \leq  {\cal C}_{\frac{1}{2}}(l, \lam)^{2u} \;.\]
\end{lemma}

\begin{lemma}
\label{lem:complamn}
If $\lam_n$ is defined by \eqref{def:lamchoice}\;and if
\[ l_n \geq n^{\beta}\; \qquad \beta > \frac{\gamma+1}{2r+1+\gamma} \]
one has with probability at least $1-\eta$ 
\[ {\cal C}_{\frac{1}{2}}(l_n, \lam_n) \leq  C \log(2\etainv)\sqrt{\lam_n}  \;, \]
provided $n$ is sufficiently large.
\end{lemma}

\begin{proof}[Proof of Lemma \ref{lem:complamn}]
Using Proposition 3 in \cite{RudCamRos15} one has with probability at least $1-\eta$
\begin{align*}
{\cal C}_{\frac{1}{2}}(l, \lam_n) &\leq \sqrt{\lam_n} \;||(T_{\bx_l }+\lam_n)^{-1}(T+\lam_n)||^{\frac{1}{2}} \\
&\leq  C \log(2\etainv)\sqrt{\lam_n} \; \cB^{\frac{1}{2}}_l( T, \lam_n) \;.
\end{align*}
Recall that $\cN(T ,\lam) \leq C_b \lam^{-\frac{1}{b}}$, implying 
\[ \cB_l(T, \lam_n) \leq C\paren{ 1 + \paren{ \frac{2}{l\lam_n} + \sqrt{\frac{\lam_n^{-\gamma}}{l\lam_n}} }^2 } \;.\]
Straightforward calculation shows that 
\[ \frac{2}{l_n\lam_n} = o(1)\;, \quad \mbox{if}\;\; l_n \geq n^{\beta} \;,\; \beta > \frac{1}{2r+1+\gamma}  \]
and
\[  \sqrt{\frac{\lam_n^{-\gamma}}{l_n\lam_n}}= o(1)\;, \quad \mbox{if}\;\; l_n \geq n^{\beta} \;,\; \beta > \frac{\gamma+1}{2r+1+\gamma}\;. \]
Thus, ${\cal C}_{\frac{1}{2}}(l_n, \lam_n) \leq  C \log(2\etainv)\sqrt{\lam_n}$,\; with probability at least $1-\eta$.
\end{proof}


\begin{prop}[Expectation Approximation- and Computational Error KRLS-Nystr\"om]
\label{prop:approx_nys}
Assume that
\[  l_n \geq n^{\beta} \;, \qquad  \beta > \frac{\gamma+1}{2r+1+\gamma}  \]
and $(\lam_n)_n$ is chosen according to \eqref{def:lamchoice}. If $n$ is sufficiently large 
\[ 
\mbe \Big[\big\|    \sqrt{T} g_{\lam_n , l_n}(T_{\bx})(T_{\bx} \fo - \fo )   \big\|_{L^2(\nu)}^2\Big]^{\frac{1}{2}} 
        \leq   C \; a_n \;,
 \]
where $C $ does not depend on $(\sigma , M, R) \in \mbr_+^3$.
\end{prop}

\begin{proof}[Proof of Proposition \ref{prop:approx_nys}]
Using that $||T^{-r}f_\rho||_{\cH}\leq R$ one has for any $\lam \in (0,1]$ 
\begin{equation}
\label{eq:dec_nervig}
T_2 \leq C R\;(\;(a) + (b) + (c)\;) \;,
\end{equation}
with
\[  (a) = ||\sqrt{ T} (Id-VV^*)  T^r  || \;, \qquad (b) = \lam || \sqrt{ T}g_{\lam , l}(T_{\bx}) T^r || \]
and
\[  (c) = || \sqrt{ T} g_{\lam , l}(T_{\bx}) ( T_{\bx} + \lam )(Id-VV^*) T^r|| \;.  \]
Since $(Id-VV^*)^2 = (Id-VV^*)$ we obtain by Lemma \ref{lem:comp_error}
\[ (a) \leq {\cal C}_{\frac{1}{2}}(l, \lam) \;{\cal C}_r(l, \lam) \leq {\cal C}_{\frac{1}{2}}(l, \lam)^{2r+1}  \;.\]
Furthermore, using \eqref{Guo:estimate_nips}\;, with probability at least $1-\frac{\eta}{2}$
\begin{align*}
(b) &\leq   C\log^{2}(8\etainv)\lam \cB^{\frac{1}{2}+r}_n( T, \lam) \;|| ( T_{\bx} + \lam )^{1/2} g_{\lam , l}(T_{\bx}) ( T_{\bx} + \lam )^{r}|| \\
&\leq C\log^{2}(8\etainv)\lam^{\frac{1}{2}+r} \cB^{\frac{1}{2}+r}_n( T, \lam)\;, 
\end{align*} 
by again using Proposition 6 in  \cite{RudCamRos15}. 

The last term gives with probability at least $1-\frac{\eta}{2}$
\begin{align*}
(c) &\leq C\log(8\etainv) || (T_{\bx} + \lam )^{1/2} g_{\lam , l}(T_{\bx}) ( T_{\bx} + \lam )|| \;{\cal C}_r(l, \lam)  \\
&\leq  C\log(8\etainv) \sqrt{\lam }\; {\cal C}_{\frac{1}{2}}(l, \lam)^{2r}\;.
\end{align*} 
Combining the estimates for $(a)$, $(b)$ and $(c)$ gives 
\[ T_2 \leq CR\log^2(8\etainv)\paren{ {\cal C}_{\frac{1}{2}}(l, \lam)^{2r+1} + \lam^{\frac{1}{2}+r} \cB^{\frac{1}{2}+r}_n( T, \lam) + \sqrt{\lam }\; {\cal C}_{\frac{1}{2}}(l, \lam)^{2r} } \;.\]
We now choose $\lam_n$ according to \eqref{def:lamchoice}\;. 
Notice that by Lemma \ref{lem:blam} one has $\cB_n(T, \lam_n) \leq C$ for any $n$ 
sufficiently large. 
Applying Lemma \ref{lem:complamn} we obtain, with probability at least $1-\eta$
\[ T_2 \leq C\log^2(8\etainv)R \lam_n^{r + \frac{1}{2}} \;,\]
provided $n$ is sufficiently large and 
\[  l_n \geq n^{\beta} \;, \qquad  \beta > \frac{\gamma+1}{2r+1+\gamma}\;. \]
The result follows from integration by applying Lemma \ref{lem:integration_nips} and recalling that $a_n = R\lam_n^{r+\frac{1}{2}}$\;.
\end{proof}

\noindent
With these preparations we can now prove the main result of Section \ref{sec:KRR_Nys}. 

\begin{proof}[Proof of Theorem \ref{theo:plain_nys}]
The proof easily follows by combining Proposition \ref{prop:sample_nys} and Proposition \ref{prop:approx_nys}\;. In particular, the estimate 
for the sample error by choosing $\lam=\lam_n$ follows by recalling that $\cN( T, \lam_n) \leq C_\gamma \lam_n^{-\gamma}$, 
by definition of $(a_n)_n$ in Theorem \ref{theo:plain_nys}\;, by Lemma \ref{lem:blam} and by 
\[  \frac{M}{n\lam_n}=o\paren{ \sigma\sqrt{ \frac{\lam_n^{-\gamma}}{n\lam_n} } } \;. \]
\end{proof}


\section{Proofs of Section \ref{sec:localnysation}}

Following the lines in the previous sections we divide the error analysis in bounding the Sample error, Approximation error and Computational error.

\begin{prop}[Sample Error]
\label{prop:sample_localnys}
Let $\lam_n$ be defined as in \eqref{def:lamchoice}. We have  
\[ \mbe\Big[ \big\|  \sum_{j=1}^m g_{\lam_n,l}( T_{\bx_j})( T_{\bx_j}\hat f_j -  S^*_{\bx_j}\by_j)  \big\|_{L^2(\nu)}^2  \Big] 
\leq C R^2\; \paren{\frac{\sigma^2}{R^2n}}^{\frac{2(r+\frac{1}{2})}{2r+1+\gamma}} \;, \]
where $n$ has to be chosen sufficiently large, i.e.
\begin{equation*}
n>  C_{\sigma, R, \gamma, r}\; m^{1+ \frac{\gamma+1}{2r+1+\gamma}}\;, 
\end{equation*}
for some $C_{\sigma, R, \gamma, r}< \infty$. 
Moreover, $C$ does not depend on the model parameter $\sigma, M,R \in \mbr^3_+$.
\end{prop}

\begin{proof}[Proof of Proposition \ref{prop:sample_localnys}]
Applying Proposition \ref{prop:sample_nys} we obtain
\begin{small}
\begin{align*}
\mbe\Big[ \big\|  \sum_{j=1}^m g_{\lam,l}(T_{\bx_j})( T_{\bx_j}\hat f_j -  S^*_{\bx_j}\by_j)  \big\|_{L^2(\nu)}^2  \Big] 
 &= \sum_{j=1}^m p_j\mbe\Big[ \big\|g_{\lam,l}( T_{\bx_j})( T_{\bx_j}\hat f_j -  S^*_{\bx_j}\by_j) \big\|_{L^2(\nu_j)}^2  \Big] \nonumber \\  
  &\leq  C 
     \sum_{j=1}^m p_j\;
           \cB_{\frac{n}{m}}^2( T_j ,\lam)\lam\;  \left( \frac{Mm}{n \lambda} + \sigma\sqrt{\frac{ {m{\cal N}( T_j, \lam)}}{{n \lam }}}\right)^2 \;.
\end{align*}
\end{small}
Arguing as in the proof of Theorem \ref{theo:fixed_partition}, using Lemma \ref{lem:fixed1}, implies the result.
\end{proof}


\begin{prop}[Approximation and Computational Error]
\label{prop:approx_localnys}
Let $\lam_n$ be defined by \eqref{def:lamchoice}.  Assume the number of subsampled points satisfies $l_n \geq n^{\beta}$ with 
\[\beta > \frac{\gamma+1}{2r+\gamma+1} \;. \]
Then 
\[\mbe\Big[ \big\|  \sum_{j=1}^{m} g_{\lam_n ,l_n}( T_{\bx_j})( T_{\bx_j} f_j - f_j  )  \big\|_{L^2(\nu)}^2  \Big] 
   \leq C R^2\; \paren{\frac{\sigma^2}{R^2n}}^{\frac{2(r+\frac{1}{2})}{2r+\gamma+1}} \;, \]
where $C$ does not depend on the model parameter $\sigma, M, R$.
\end{prop}

\begin{proof}[Proof of Proposition \ref{prop:approx_localnys}]
For proving this Proposition we combine techniques from both the partitioning and subsampling approach. More precisely:
 
\begin{small}
\begin{align*}
\mbe\Big[ \big\|  \sum_{j=1}^m g_{\lam_n ,l_n}( T_{\bx_j})( T_{\bx_j} f_j - f_j )  \big\|_{L^2(\nu)}^2  \Big] 
 &= \sum_{j=1}^m p_j\mbe\Big[ \big\|g_{\lam_n ,l_l}( T_{\bx_j})( T_{\bx_j} f_j - f_j) \big\|_{L^2(\nu_j)}^2  \Big] \nonumber \\  
  &= \sum_{j=1}^m p_j\mbe\Big[ \big\| \sqrt{ T_j} g_{\lam_n ,l_n}( T_{\bx_j})( T_{\bx_j} f_j - f_j)  \big\|_{\hat \cH_j}^2  \Big]  \;.  
\end{align*}
\end{small}

We shall decompose as in \eqref{eq:dec_nervig}, with  $T$ replaced by $T_j$ and $T_{\bx}$ replaced by $T_{\bx_j}$,  
 \[   ||\sqrt{\bar T_j} g_{\lam_n ,l_n}(T_{\bx_j})(T_{\bx_j}\hat f_j - f_j)||_{\hat \cH_j} \leq C R\;(\;(a) + (b) + (c)\;) = (*)\;. \]
Following the lines of the proof of Proposition \ref{prop:approx_nys} leads to an upper bound (with probability at least $1-\eta$) for the rhs of the last inequality, which is 
\begin{align*}
(*) &\leq CR\log^2(8\etainv)\paren{ {\cal C}_{\frac{1}{2}}(l, \lam_n)^{2r+1} + \lam_n^{\frac{1}{2}+r} \cB^{\frac{1}{2}+r}_{\frac{n}{m}}(T_j, \lam_n) + \sqrt{\lam_n }\; {\cal C}_{\frac{1}{2}}(l, \lam_n)^{2r} } \\
 &\leq CR\log^2(8\etainv) \lam_n^{r+\frac{1}{2}}\paren{ \cB^{2r+1}_{l}( T_j, \lam_n) +\cB^{r+\frac{1}{2}}_{\frac{n}{m}}( T_j, \lam_n) +\cB^{2r}_{l}( T_j, \lam_n )  } \;.
\end{align*}

Thus, by integration and since $r \leq \frac{1}{2}$

\begin{small}
\begin{align*}
\mbe\Big[ \big\|  \sum_{j=1}^m g_{\lam_n ,l_n}( T_{\bx_j})( T_{\bx_j}f_j - f_j )  \big\|_{L^2(\nu)}^2  \Big] 
 &\leq CR^2\lam_n^{2(r+\frac{1}{2})} \sum_{j=1}^m p_j\paren{ \cB^{4}_{l}( T_j, \lam_n) +\cB^{2}_{\frac{n}{m}}( T_j, \lam_n) +\cB^{2}_{l}( T_j, \lam_n )  } \;.
\end{align*}
\end{small}

Note that by Lemma \ref{lem:fixed1}, if 
\begin{equation}\label{eq:final_alpha}
n\geq C_{\sigma, R, \gamma, r} m^{1+\frac{\gamma+1}{2r}}
\end{equation}
we have 
\begin{align*}
  \cB_{\frac{n}{m}}( T_j, \lam_n) &=  \left[1 + \left( \frac{2m}{n\lam_n} + \sqrt{\frac{m_n\cN( T_j ,\lam_n)}{n\lam}}\right)^2 \right] \\
&\leq  C\left[1 +   \paren{ \frac{2m}{n\lam_n} } + \paren{ \frac{m\cN( T_j ,\lam_n)}{n\lam}}    \right]  \\
&\leq C  \;.
\end{align*}  
Moreover, since $\cN(T_j, \lam_n) \leq \cN(T, \lam_n/p_j)$, by Assumption \ref{ass:one}, 2. and since $p_j\leq 1$
\[ \cB_{l_n}( T_j, \lam_n) \leq   1+\paren{  \frac{2}{l_n\lam_n} + \sigma\sqrt{  \frac{ \lam_n^{-\gamma}}{l_n \lam_n}  }   }^2  \;. \]
Straightforward calculation shows that 
\begin{equation*}
 \frac{2}{l_n\lam_n} = o(1)\;, \quad \mbox{if}\;\; l_n \geq n^{\beta'} \;,\; \beta' > \frac{1}{2r+\gamma+1}  
\end{equation*}
and
\begin{equation}
\label{eq:beta2}
  \sqrt{\frac{\lam_n^{-\gamma}}{l_n\lam_n}}= \cO(1)\;, 
\quad \mbox{if}\;\; l_n \geq n^{\beta'} \;,\; \beta' \geq \frac{\gamma+1}{2r+\gamma+1}\;. 
\end{equation}
Thus, \eqref{eq:beta2} ensures $ \cB_{l_n}( T_j, \lam_n)  = \cO(1)$. Finally, on each local set we have the requirement 
$l_n  \lesssim \frac{n}{m_n}$, which is implied by 
\[ l_n  \lesssim n^{1-\alpha} \sim n^{\frac{\gamma + 1}{2r+\gamma +1}} \; .\] 
Together with \eqref{eq:beta2} we get a sharp bound 
\[  l_n \sim n^{\frac{\gamma + 1}{2r+\gamma +1}} \;. \]
\end{proof}


\section{Probabilistic Inequalities}

In this section we recall some well-known probabilistic inequalities.

\begin{prop}[\cite{BlaMuc16}]
\label{Geta1_nips}  
For $n \in \mbn$, $\lambda \in (0,1]$ and  $\eta \in (0,1]$, one has with probability at least $1-\eta$\,:
\begin{equation*}
\big\| ( T +  \lam)^{-\frac{1}{2}}\;\left( T_{\bx}f_{\rho} -  S_{\bx}^*\by \right)\big\|_{\cH}\;  \leq \; 
2\log(2\eta^{-1})  \left( \frac{M}{n\sqrt{\lam}} + \sigma\sqrt{\frac{ {\cal N}( T, \lam)}{ n}} \right)\;.
\end{equation*}
\end{prop}

\begin{prop}[\cite{BlaMuc16}, Proposition 5.3]
\label{Geta2_nips}
For any $\lambda \in (0,1]$ and $\eta \in (0,1)$ one has with probability at least $1-\eta $\,:
\begin{equation*}
\norm{( T+ \lam)^{-1}( T- T_{\bx}) }_{HS} \; \leq 
2\log(2\eta^{-1}) \left( \frac{2}{n \lam} + \sqrt{\frac{\cN( T,\lam)}{n\lam }}  \right)\; .
\end{equation*}
\end{prop}

\begin{prop}[\cite{guo17}]
\label{prop:Guo_nips}
Define
\begin{equation}
\label{def:blam_nips}
  \cB_{ n}( T, \lam) :=  \left[1 + \left( \frac{2}{n\lam} + \sqrt{\frac{\cN( T ,\lam)}{n\lam}}\right)^2 \right] 
\end{equation}  
For any $\lam >0$,  $\eta \in (0,1]$, with probability at least $1-\eta$ one has 
\begin{equation}
\label{Guo:estimate_nips}
 \norm{ ( T_{\bx}+\lam)^{-1}( T+\lam) } \leq 
     8\log^2(2\etainv)   \cB_n( T, \lam) \;.
\end{equation}
\end{prop}

\begin{lemma}
\label{lem:fixed1}
Let $m \in \mbn$ and  $\lam_n$ be defined by \eqref{def:lamchoice}.  Then for any $j\in [m]$ and $n > n_0$
\[ \cB_{\frac{n}{m}}( T_j, \lam_n ) \leq 2 \;.  \] 
Here, $n_0$ depends on the number $m$ of subsets and the model parameter $R, \sigma, \gamma, r$ and is explicitly given in \eqref{def:n0}.
\end{lemma}

\begin{proof}[Proof of Lemma \ref{lem:fixed1}]
Recall that we assume $\cN(T, \lam)\leq C_\gamma \lam^{-\gamma}$, for some $b\geq 1$, $C_\gamma<\infty$. Thus, by Lemma \ref{lem:eff_dim_sum}\; we have 
for any  $j \in [m]$
\[ \cN(T_j ,\lam) \; \leq \; \cN( T ,\lam/p_j) \; \leq  \; C_\gamma \; p_j^{\gamma} \; \lam^{-\gamma} \]
and thus 
\[  \frac{m\cN( T_j ,\lam_n)}{n\lam_n} \leq C_\gamma p_j\frac{m}{n} \lam_n^{-(1+\gamma)}  < \frac{1}{2} \;,\]
provided  
\[ n  > (2C_\gamma p_jm)^{\frac{2r+\gamma+1}{2r}}\;\paren{\frac{R}{\sigma}}^{\frac{2(\gamma+1)}{2r}} \;. \]
Moreover,
\[  \frac{2m}{n\lam_n} < \frac{1}{2}\;, \]
provided
\[  n>  (4m)^{\frac{2r+\gamma+1}{2r+1}}\;\paren{\frac{R}{\sigma}}^{\frac{2}{2r+\gamma}} \;. \]
Finally, setting $p_{max}=\max(p_1, ..., p_m)$, if 
\begin{equation}
\label{def:n0}
n> n_0 := (4m)^{\frac{2r+\gamma+1}{2r}} \;\max\paren{ \;\paren{R/\sigma}^{\frac{2}{2r+\gamma}}  \;, (p_{max}\;C_\gamma)^{\frac{2r+\gamma+1}{2r}} \paren{R/\sigma}^{\frac{2(\gamma+1)}{2r}} \;} 
\end{equation}
we have
\[ \cB_{\frac{n}{m}}(T_j, \lam_n) \leq 1+ \paren{ \frac{1}{2} + \frac{1}{2} }^2 = 2 \;,  \]
uniformly for any $j \in [m]$.
\end{proof}

\begin{lemma}
\label{lem:blam}
If $\lam_n$ is defined by \eqref{def:lamchoice} 
\[   \cB_{n}( T, \lam_n) \leq 2 \;, \]
provided $n$ is sufficiently large. 
\end{lemma}

\begin{proof}[Proof of Lemma \ref{lem:blam}]
The proof is a straightforward calculation using Definition \eqref{def:lamchoice} and recalling 
that $\cN( T ,\lam) \leq C_\gamma \lam^{-\gamma}$\;.
\end{proof}




\section{Miscellanea}

\begin{prop}[Cordes Inequality,\cite{Bat97}, Theorem IX.2.1-2]
\label{prop:Cordes}
Let $A,B$ be two bounded, self-adjoint and positive operators on a Hilbert space. Then for any $s\in[0,1]$:
\begin{equation}
\norm{A^sB^s} \leq \norm{AB}^s\,.
\end{equation}
\end{prop}

\begin{lemma}
\label{lem:integration_nips}
Let $X$ be a non-negative random variable with $\mbp[ X > C\log^u(k\etainv)] < \eta$ for any $\eta \in (0,1]$. Then
$\mbe[ X ] \leq \frac{C}{k}u\Gamma (u)$. 
\end{lemma}

\begin{proof}
Apply $\mbe[X] = \int_0^\infty \mbp[X > t] dt$. 
\end{proof}

\end{document}